\documentclass[a4paper,12pt]{amsart}
\usepackage{amssymb}
\usepackage{mathrsfs}
\usepackage{color}
\usepackage{epsfig}
\usepackage{graphicx}
\usepackage[all]{xy}

\theoremstyle{plain}
\newtheorem{theorem}{Theorem}[section]
\newtheorem{prop}[theorem]{Proposition}

\newtheorem{lemma}[theorem]{Lemma}

\theoremstyle{definition}
\newtheorem{remark}[theorem]{Remark}

\newtheorem{definition}[theorem]{Definition}
\newtheorem{example}[theorem]{Example}

\newtheorem{assumption}[theorem]{Assumption}

\newcommand{\C}{\mathbb{C}}
\newcommand{\R}{\mathbb{R}}
\newcommand{\Z}{\mathbb{Z}}

\newcommand{\vep}{\varepsilon}
\newcommand{\loc}{\rm loc}

\renewcommand{\tilde}{\widetilde}
\renewcommand{\setminus}{\smallsetminus}
\newcommand{\nin}{/\kern-2.1ex\in}

\def\<{\left\langle}
\def\>{\right\rangle}

\def\Ker{\operatorname{Ker}}
\def\ind{\operatorname{ind}}

\numberwithin{equation}{section}

\title[$S^1$-equivariant local index]
{$S^1$-equivariant local index and transverse index 
\\ for non-compact symplectic manifolds}

\author[H. Fujita]{Hajime Fujita$^1$}

\thanks{$^1$Partly supported by Grant-in-Aid for Young Scientists (B) 23740059.}

\address{Department of Mathematical and Physical Sciences \\ Japan Women's University \\ 2-8-1 Mejirodai, Bunkyo-ku Tokyo, 112-8681 Japan}
  \email{fujitah@fc.jwu.ac.jp}

\begin{document}

\maketitle

\begin{abstract}
We define an $S^1$-equivariant index 
for non-compact symplectic manifolds with Hamiltonian $S^1$-action. 
We use the perturbation by Dirac-type operator along the $S^1$-orbits. 
We give a formulation and a proof 
of quantization conjecture for this $S^1$-equivariant index. 
We also give comments on the relation 
between our $S^1$-equivariant index and 
the index of transverse elliptic operators. 
\end{abstract}


\section{Introduction}

In \cite{Fujita-Furuta-Yoshida1}, Furuta, Yoshida and the author gave a formulation 
of index theory of Dirac-type operator on open manifolds 
using  torus fibration and 
the perturbation by Dirac-type operator along fibers. 
In \cite{Fujita-Furuta-Yoshida2}  a refinement of 
it for a family of torus bundles with some compatibility conditions was given.  
In \cite{Fujita-Furuta-Yoshida3} the authors used 
equivariant version of them to give a geometric proof of 
quantization conjecture for Hamiltonian torus action on 
closed symplectic manifolds. 
In this paper we give a formulation of $S^1$-equivariant index theory 
for non-compact symplectic manifold with Hamiltonian $S^1$-action based on the framework  
of \cite{Fujita-Furuta-Yoshida1}.  
The resulting index is a homomorphism from $R(S^1)$ to $\Z$, 
and if the manifold is closed, then the $S^1$-equivariant index 
coincides with the Riemann-Roch character as a functional on $R(S^1)$.

We use a perturbation by the Dirac-type operator along 
$S^1$-orbits. On the other hand Braverman~\cite{Braverman} gave an 
index theory on open manifolds based on a perturbation 
by the vector field induced from certain equivariant map, 
e.g., moment map. 
His index theory realizes the index of transverse 
elliptic operators developed by Atiyah~\cite{Atiyah1} and Paradan-Vergne~\cite{ParadanVergne}. 
Both our construction and 
Braverman's construction use perturbation 
by operators along the orbits, and hence, they have conceptual similarity. 
We show that they are equal for the proper moment map case. 
We also show that they have different nature. 
In fact we will give an example in this paper which shows the difference.

Yoshida~\cite{Yoshida3} gave an expository article on the $S^1$-equivariant 
version of the index theory developed in \cite{Fujita-Furuta-Yoshida1, 
Fujita-Furuta-Yoshida2, Fujita-Furuta-Yoshida3}. 
The equivariant local index considered in \cite{Yoshida3} is a
straightforward generalization for the equivariant setting. 
Then the index in \cite{Yoshida3} is a finite dimensional object. 
On the other hand the transverse index in 
\cite{Braverman}, \cite{ Ma-Zhang, Ma-Zhangarxiv}, \cite{Paradan2} and  \cite{Vergne3} has infinite dimensional nature. 
In particular the index in \cite{Yoshida3} does not coincide with 
the transverse index.  
The equivariant index in the present paper is an 
another kind of generalization of \cite{Fujita-Furuta-Yoshida1} for the equivariant setting, 
which has an infinite dimensional nature.

In our construction it is straightforward to 
give a formulation and a proof of quantization conjecture 
for non-compact symplectic manifolds with Hamiltonian $S^1$-action. 
Vergne~\cite{Vergne3} proposed a quantization conjecture 
for non-compact symplectic manifolds, 
which was proved by Ma-Zhang~\cite{Ma-Zhang,Ma-Zhangarxiv} and Paradan gave a
new proof in \cite{Paradan2}.  
Vergne's conjecture is based on the index theory of transverse elliptic operators. 
Ma and Zhang showed Vergne's conjecture under weaker assumption, the properness of the moment map, using Braverman's index theory. 
We do not assume neither compactness of the fixed point set 
nor properness of the moment map as in \cite{Ma-Zhang,Ma-Zhangarxiv},  \cite{Paradan2} and \cite{Vergne3}.
We only assume that the inverse image of 
each integer point is compact.

This paper is organized as follows. 
In Section~2, we give a brief review of the construction in 
\cite{Fujita-Furuta-Yoshida1} 
to define an $S^1$-equivariant index $\ind_{S^1}(X, V)$. 
In Section~3, we apply the construction in Section~2 
to the symplectic geometry case. 
We define an $S^1$-equivariant local Riemann-Roch number $RR_{S^1, {\loc}}(M,L)$ 
for Hamiltonian $S^1$-action on non-compact symplectic manifold. 
In Section~4 we give a quantization conjecture for non-compact symplectic 
manifold with a Hamiltonian $S^1$-action. 
In Section~5, we give comments 
on relation between our equivariant index and that 
developed by Braverman~\cite{Braverman} and 
Ma-Zhang~\cite{Ma-Zhangarxiv}. 
In Appendix~A we give some details of the explicit computation of 
the kernel of a family of perturbed Dirac operators on the cylinder. 
The family contains the perturbations in this paper and that in 
\cite{Braverman}, and it shows the difference between two equivariant indices. 

\subsection{Notations}
\begin{itemize}
\item For each $n\in \Z$ let $\C_{(n)}$ be the complex line 
with the standard action of the circle group $S^1$ of weight $n$. 
\item Let $\rho$ be a representation space of $S^1$. 
For each $n\in \Z$ we denote by $\rho^{(n)}$ the multiplicity of the 
weight $n$ representation in $\rho$, i.e., we put 
$$
\rho^{(n)}:=\dim\left({\rm Hom}_{S^1}(\C_{(n)},\rho)\right).
$$ 
We will also use the same notation for elements in the representation ring 
$R(S^1)$. 
\end{itemize}

\section{Definition of the $S^1$-equivariant index}\label{boundary}\label{secboundary}
In this section we give a brief review of the construction in 
\cite{Fujita-Furuta-Yoshida1} 
to define an index $\ind_{S^1}(X, V)$.

\subsection{Setting} 
Let $X$ be a non-compact Riemannian manifold. 
Let $W=W^+\oplus W^-$ be a $\Z/2$-graded $Cl(TX)$-module bundle 
with the Clifford multiplication $c$. 
Suppose that the circle group $S^1$ acts on 
$X$ in an isometric way and the action lifts to $W$ so that it commutes 
with $c$. 
We assume that there exists an open subset $V$ of $X$ 
which satisfies the following assumption.

\begin{assumption}\label{n-acyclic}
\begin{enumerate}
\item The complement $X\setminus V$ is compact. 
\item $S^1$ acts on $V$ without fixed points. 
\item There exists an $S^1$-equivariant formally self-adjoint operator 
$D_{S^1}:\Gamma(W|_{V})\to \Gamma(W|_{V})$ which 
satisfies the following conditions. 
\begin{enumerate}
\item $D_{S^1}$ contains only the derivatives along the $S^1$-orbits, 
and its restriction to each orbit is a Dirac-type operator 
along the orbit. 
\item 
For each tangent vector $u$ which is normal to 
the $S^1$-orbit, $D_{S^1}$ anti-commutes with 
the Clifford multiplication of $\tilde u$ : 
$$
c(\tilde u)\circ D_{S^1}+D_{S^1}\circ c(\tilde u)=0, 
$$where $\tilde u$ is the vector field along the $S^1$-orbit 
which is obtained by $u$ and the $S^1$-action. 
\end{enumerate}
\item 
For all $x\in V$ the kernel of the restriction of $D_{S^1}$ 
to the orbit $S^1\cdot x$ is trivial,  
i.e., $\ker(D_{S^1}|_{S^1\cdot x})=0$. 
\end{enumerate}
\end{assumption}

If the data $(X, V, W, D_{S^1})$ satisfies these conditions, 
then we call $(X, V, W, D_{S^1})$ {\it acyclic}.  

\subsection{Definition of $\ind_{S^1}(X,V)$}
Following the procedure as in \cite{Fujita-Furuta-Yoshida1}, 
we can define an $S^1$-equivariant index 
${\rm ind}_{S^1}(X,V)\in R(S^1)$ for the acyclic data $(X, V, W, D_{S^1})$. 
Let us recall the definition. 
By using an $S^1$-invariant proper function 
we can deform the end $V$ of $X$ into a complete Riemannian manifold $\hat X$, 
for instance, with cylindrical end $\hat V$ in an $S^1$-equivariant way. 
Namely let $f:X\to \R$ be an $S^1$-invariant 
smooth function and $c$ its regular value such that 
$f^{-1}((-\infty, c])$ is compact and contains $X\setminus V$. 
Then $\hat X$ is obtained by attaching the cylinder 
$f^{-1}(c)\times \R_+$ to $f^{-1}((-\infty, c])$. 
We also deform $W$ and $D_{S^1}$ into $\hat W$ and $\hat D_{S^1}$,  
which have translational invariance on the end $\hat V$. 
Let $D$ be an $S^1$-equivariant Dirac-type operator on $\Gamma(\hat W)$ which is translationally invariant on $\hat V$. 
Let $\rho_{V}$ be an $S^1$-invariant cut-off function such that $\rho_V=0$ on $X\setminus V$ and $\rho_V=1$ on the end of $\hat X$. 
For $t\geq  0$ we consider an analog of Witten's deformation  
$D_t:=D+t\rho_V\hat D_{S^1}$. 
We can show that it gives a Fredholm operator on $L^2(\hat W)$ for any $t\gg 1$. 
One of the key is the following estimate. 
See also \cite[Lemma~5.2 and Lemma~5.4]{Fujita-Furuta-Yoshida1}. 


\begin{lemma}\label{vanishing}
There exists a positive constant $T$ such that for any $t>T$ we have 
$$
\|D_ts\|_{L^2}\geq \|s\|_{L^2}
$$ for any compactly supported section $s$ of $\hat W$ whose support 
is contained in $\hat V$. 
\end{lemma}
\begin{proof}
Let $s$ be a compactly supported section $s$ of $\hat W$ whose support 
is contained in $\hat V$. 
By integration by parts we have 
\begin{eqnarray*}
\|D_ts\|_{L^2}^2&=&\int_{\hat V}(D_ts,D_ts)=\int_{\hat V}({D_t}^2s,s)\\
&=&\int_{\hat V}(D^2s,s)+t\int_{\hat V}((D\hat D_{S^1}+\hat D_{S^1}D)s,s)+
t^2\int_{\hat V}((\hat{D}_{S^1})^2s,s)\\
&\geq& t\int_{\hat V}((D\hat D_{S^1}+\hat D_{S^1}D)s,s)+
t^2\int_{\hat V}((\hat D_{S^1})^2s,s). 
\end{eqnarray*}
The anti-commutativity 
(Assumption~\ref{n-acyclic}(3)-(b)) implies that 
the anti-commutator $D\hat D_{S^1}+\hat D_{S^1}D$ is a differential operator 
along the $S^1$-orbits. (See \cite[Lemma~5.10]{Fujita-Furuta-Yoshida1}.) 
By Assumption~\ref{n-acyclic}(4) and 
a priori estimate, there exists a positive constant $C_1$ and $C_2$ 
such that  
$$
C_1\int_{\rm orbit}((\hat D_{S^1})^2s,s)\geq
\left|\int_{\rm orbit}((D\hat D_{S^1}+\hat D_{S^1}D)s,s)\right|
$$ and 
$$ 
C_2\int_{\rm orbit}((\hat D_{S^1})^2s,s)\geq\int_{\rm orbit}(s,s)
$$for all $S^1$-orbits in $\hat V$. 
Note that since $\hat V$ has cylindrical end we may assume that 
$C_1$ and $C_2$ do not depend on the choice of orbits. 
Then we have 
$$
\|D_{t}s\|^2_{L^2}\geq (t^2-C_1t)\int_{\hat V}((\hat D_{S^1})^2s,s)\geq C_2^{-1}(t^2-C_1t)\|s\|_{L^2}^2,  
$$and hence, by taking a positive number $T$ as $T^2-C_1T\geq C_2$ and 
$T\geq C_1/2$, we obtain the required inequality. 
\end{proof}

\begin{remark}\label{remvanishing}
Note that since the principal symbol of $D_t$ is given by 
a combination of the Clifford action, it has finite propagation speed. 
It is well-known that the finite propagation speed implies the essentially self-adjointness. See \cite{Chernoff} for example. 
In particular the inequality in Lemma~\ref{vanishing} holds 
for any $L^2$-sections of $\hat W$ whose supports are contained in $\hat V$, 
and if $\hat X=\hat V$, then we have the vanishing of the space of $L^2$-solutions, i.e.,  
$\ker_{L^2}D_t=\{0\}$.  
\end{remark}

Since $D$, $\rho_{V}$ and $\hat D_{S^1}$ are 
$S^1$-equivariant $S^1$ acts on the space of $L^2$-solutions of $D_ts=0$, we can define the Fredholm index $$\ind_{S^1}(\hat X, \hat V):=\ind(D_t)=\ker(D_t|_{L^2(W^+)})-\ker(D_t|_{L^2(W^-)})$$ as an element of the equivariant $K$-group $K_{S^1}({\rm pt})=R(S^1)$. The index is invariant under the 
continuous deformation of the given data, it does not depend on the choice of $t\gg 1$. Moreover we have the following. 

\begin{prop}
The Fredholm index of $D_t$ does not depend on the choice of 
the completion $\hat X$ of $X$. 
\end{prop}
\begin{proof}
Suppose that there are two completions $\hat X_1$ and $\hat X_2$. 
Namely  for $i=1,2$ let $f_i:X\to \R$ be an $S^1$-invariant 
smooth function and $c_i$ its regular value such that 
$f_i^{-1}((-\infty, c_i])$ is compact and contains $X\setminus V$. 
Then $\hat X_i$ is obtained by attaching the cylinder 
$f_i^{-1}(c_i)\times \R_+$ to $f_i^{-1}((-\infty, c_i])$ whose end is $\hat V_i:=(V\cap f_i^{-1}((-\infty, c_i]))\cup f_i^{-1}(c_i)\times \R_+$. 
It is enough to show for the case 
$f_1^{-1}((-\infty, c_1])\subset f_2^{-1}((-\infty, c_2])$. 
In this case, by the gluing formula of indices (\cite[Lemma~4.8]{Fujita-Furuta-Yoshida2}), 
the index of $\hat X_2$ is the sum of 
the indices of $\hat X_1$ and the completion $\hat X_{12}$ of 
$f_2^{-1}((-\infty, c_2])\setminus f_1^{-1}((-\infty, c_1))$ with an end $\hat V_{12}$ : 
$$
\ind(\hat X_2, \hat V_2)=\ind(\hat X_1,\hat V_1)+\ind(\hat X_{12}, \hat V_{12}). 
$$
Since $f_2^{-1}((-\infty, c_2])\setminus f_1^{-1}((-\infty, c_1))$ is 
contained in $V$, we may take $\hat V_{12}=\hat X_{12}$. Then the space of $L^2$-solutions on $\hat X_{12}$ vanishes by Remark~\ref{remvanishing}. 
It implies that $\ind(\hat X_{12}, \hat V_{12})=0$ and $\ind(\hat X_2, \hat V_2)=\ind(\hat X_1,\hat V_1)$. 
\end{proof}

\begin{definition}\label{equivlocind}
Let $T$ be the positive constant as in Lemma~\ref{vanishing}. 
We define the equivariant index $\ind_{S^1}(X,V)\in K_{S^1}({\rm pt})=R(S^1)$ 
to be the $S^1$-equivariant Fredholm index of $D_t$ on  a completion $(\hat X, \hat V)$ for any $t>T$,  
$$
\ind_{S^1}(X,V):=\ind_{S^1}(\hat X, \hat V)=\ind_{S^1}(D_t). 
$$
\end{definition}

The equivariant index $\ind_{S^1}(X,V)$ satisfies the excision formula, 
gluing formula and product formula.

\section{$S^1$-equivariant local index for symplectic manifolds 
with Hamiltonian $S^1$-action}\label{propermomentmap}
Let $(M,\omega)$ be a (possibly non-compact) symplectic manifold with 
a pre-quantizing line bundle $(L,\nabla)$, i.e., 
$L$ is a Hermitian line bundle over $M$ and $\nabla$ is its Hermitian 
connection whose curvature form is equal to $-\sqrt{-1}\omega$. 
Suppose that the circle group $S^1$ acts on 
$(M,\omega)$ and the action lifts to $(L,\nabla)$. 
Note that for each $x\in M$ the restriction $(L,\nabla)|_{S^1\cdot x}$ is a 
flat line bundle. 
Let $\mu:M\to \R$ be the associated moment map. 
We assume the following compactness. 

\medskip

\noindent
{\bf Assumption.} For each  $n\in\Z$, the inverse image $\mu^{-1}(n)$  
is a compact subset. 

\medskip

\noindent
Take and fix an $S^1$-invariant $\omega$-compatible almost complex structure 
$J$ on $M$ so that we have the associated metric $g^J$ and 
$\Z/2$-graded $Cl(TM)$-module bundle 
$W_L:=\wedge^{\bullet}T^*M^{0,1}\otimes L$.  
We put $V$ to be the complement of the fixed point set, $V:=M\setminus M^{S^1}$. 
Let $T_{S^1}V\to V$ be the tangent bundle along 
the $S^1$-orbits, which is, by definition, a real line bundle  
over $V$. 
Let $E$ be the orthogonal complement of $T_{S^1}V\oplus J(T_{S^1}V)\cong T_{S^1}V\otimes \C$ in $TM|_V$. 
Note that $J(T_{S^1}V)\oplus E$ is the normal bundle of $T_{S^1}V$ and isomorphic to the pull-back of the tangent bundle of the quotient space $V/S^1$. 
We have isomorphisms as Hermitian vector bundles  
\begin{equation}\label{bundle isom}
\wedge^{\bullet}T^*M^{0,1}|_{V}\cong
\wedge^{\bullet}TM|_V\cong (\wedge^{\bullet}T_{S^1}V\otimes\C)\otimes (\wedge^{\bullet}E), 
\end{equation}
and hence, we have 
$$
W_L|_{V}\cong W_{S^1, L}\otimes (\wedge^{\bullet}E),
$$where $W_{S^1, L}:=\wedge^{\bullet}T_{S^1}V\otimes L|_V\cong 
(\wedge^{\bullet}T_{S^1}V)^*\otimes L|_V$ is the family of Clifford module bundle 
over the $S^1$-orbits. 
Let $\bar{D}_{S^1}:\Gamma(W_{S^1, L})\to \Gamma(W_{S^1,L})$ 
be the family of twisted de Rham operators along $S^1$-orbits 
with coefficients in $(L,\nabla)|_{V}$.
Namely $\bar{D}_{S^1}$ is the following degree-one differential operator of order-one. 

\begin{itemize}
\item $\bar{D}_{S^1}$ does not contain any differentials transverse to 
the $S^1$-orbits. 
\item For each $x\in V$ the restriction 
of $\bar{D}_{S^1}$ to the orbit ${S^1\cdot x}$ is the de Rham operator  
acting on $W_{S^1, L}|_{S^1\cdot x}$, 
the Clifford module bundle over $S^1\cdot x$ with coefficients in the flat line bundle $(L,\nabla)|_{S^1\cdot x}$. 
\end{itemize}


Since for each $x\in V$ the restriction $E|_{S^1\cdot x}$ has 
canonical flat structure induced by the $S^1$-action, $\bar D_{S^1}$ naturally induces a 
differential operator along the orbits 
$D_{S^1}:\Gamma(W_L|_V)\to \Gamma(W_L|_V)$. 
Note that for each $x\in V$ we have
$$
\ker(D_{S^1}|_{S^1\cdot x})=\ker(\bar D_{S^1}|_{S^1\cdot x})\otimes 
\wedge^{\bullet} E|_{S^1\cdot x}\cong H^{*}(S^1\cdot x, L|_{S^1\cdot x})\otimes 
\wedge^{\bullet} E|_{S^1\cdot x}
$$by Hodge theory, 
where $H^{*}(S^1\cdot x, L|_{S^1\cdot x})$ is the de Rham cohomology 
with coefficients in the flat line bundle $(L,\nabla)|_{S^1\cdot x}$ 
over the orbit. 
We first recall some basic properties. 

\begin{lemma}\label{basic properties}
\begin{enumerate}
\item For each  $x\in V$, 
the space of global parallel sections \\ $H^0(S^1\cdot x, L|_{S^1\cdot x})$ 
vanishes if and only if 
$\ker(D_{S^1}|_{S^1\cdot x})$ vanishes. 
\item For each  $x\in V$ and $n\in \Z$, 
the multiplicity $H^0(S^1\cdot x, L|_{S^1\cdot x})^{(n)}$ 
vanishes if and only if 
the multiplicity $\ker(D_{S^1}|_{S^1\cdot x})^{(n)}$ vanishes. 
\item If $H^0(S^1\cdot x, L|_{S^1\cdot x})\neq 0$, then we have $\mu(x)\in \Z$. 
In particular we have $\mu(M^{S^1})\subset \Z$. 
\item If $H^0(S^1\cdot x, L|_{S^1\cdot x})\neq 0$, then 
we have $H^0(S^1\cdot x,L|_{S^1\cdot x})=\C_{(\mu(x))}$. 
In particular if $x\in M^{S^1}$, then 
we have $L_x=\C_{(\mu(x))}$. 
\end{enumerate}
\end{lemma}
We take an $S^1$-invariant relatively compact open neighborhood $X_{\mu,n}$ 
of the compact set $\mu^{-1}(n)$ as $\mu^{-1}(n)\subset X_{\mu,n} \subset 
\mu^{-1}([n-1/2, n+1/2])$. 
We put $V_{\mu,n}:=X_{\mu,n}\setminus \mu^{-1}(n)$. 

\begin{prop}
$(X_{\mu,n},V_{\mu,n}, W_L|_{V_{\mu,n}}, D_{S^1}|_{V_{\mu,n}})$ is acyclic. 
\end{prop}
\begin{proof}
Since $\mu(V_{\mu,n})\cap \Z=\emptyset$ we have 
$(V_{\mu,n})^{S^1}=\emptyset$ and 
$\ker(D_{S^1}|_{S^1\cdot x})=0$ for all $x\in V_{\mu,n}$ by (1) and (3) in 
Lemma~\ref{basic properties}. 
For each normal tangent vector $u=u_1+u_2\in J(T_{S^1}V)_x\oplus E_x$ to the orbit $S^1\cdot x$, let $S^1\cdot u=S^1\cdot u_1+S^1\cdot u_2$ be the induced vector field along the orbit. 
Note that the vector field $J(S^1\cdot u_1)$ is parallel with respect to the Levi-Civita connection of the induced metric on $S^1\cdot x$ and 
the Clifford multiplication of $S^1\cdot u_1$ on $W|_{S^1\cdot x}$ is identified with that of $J(S^1\cdot u_1)$ on $T(S^1\cdot x)$ under the identification (\ref{bundle isom}). 
Since $S^1\cdot u_2$ is normal to $S^1\cdot x$ and the restriction $D_{S^1}|_{S^1\cdot x}$ is the de Rham operator with coefficient in $(L,\nabla)|_{S^1\cdot x}\otimes E|_{S^1\cdot x}$, it anti-commutes with the Clifford multiplication 
of $S^1\cdot u$. 
\end{proof}


We can define  the index 
$\ind_{S^1}(X_{\mu,n},V_{\mu,n})\in R(S^1)$ as in Definition~\ref{equivlocind} 
by applying the construction in Section~\ref{boundary}.  

\begin{prop}
The index $\ind_{S^1}(X_{\mu,n}, V_{\mu,n})$ does not depend on the choice 
of $X_{\mu,n}$. 
\end{prop}
\begin{proof}
Suppose that we take two relatively compact neighborhoods $X_{\mu,n}$ and $X_{\mu,n}'$ of $\mu^{-1}(n)$ with the required properties. 
It is enough to show that if $X_{\mu,n}\subset X_{\mu,n}'$, then we have 
$\ind_{S^1}(X_{\mu,n},V_{\mu,n})=\ind_{S^1}(X_{\mu,n}',V_{\mu,n}')$. 
The equality follows from the 
excision formula of $\ind_{S^1}(\cdot,\cdot)$. 
\end{proof}

\begin{definition}
For each $n\in \Z$, we put  $RR_{S^1, {\loc}}^{(n)}(M,L):=\ind_{S^1}(X_{\mu,n},V_{\mu,n})^{(n)}\in \Z$ and 
define $RR_{S^1, {\loc}}(M,L)\in {\rm Hom}(R(S^1), \Z)$ 
by putting  
$$
RR_{S^1, {\loc}}(M,L) : \C_{(n)} \mapsto 
RR^{(n)}_{S^1, {\loc}}(M,L). 
$$
We call $RR_{S^1, {\loc}}(M,L)$ the {\it $S^1$-equivariant local Riemann-Roch number.}
\end{definition}

\begin{remark}
In contrast to the equivariant index considered in \cite{Yoshida3}, 
the $S^1$-equivariant local Riemann-Roch number $RR_{S^1, {\loc}}$ 
can be defined even if $\mu(M)$ contains infinitely many integral points. 
For such case the number $RR^{(n)}_{S^1, {\loc}}(M,L)$ may be non-zero 
for infinitely many $n$. 
We do not know whether $S^1$-equivariant local Riemann-Roch number 
$RR_{S^1, {\loc}}(M,L)$ has distributional nature or not.  
\end{remark}

\begin{remark}
Using the equivariant version of the 
{\it acyclic compatible systems} in \cite{Fujita-Furuta-Yoshida2} 
it would be possible to define the equivariant local index 
$RR_{G, {\loc}}^{(\xi)}(M,L)$ and $RR_{G,loc}(M,L)$ 
for any compact torus $G$ and $\xi$ in the weight lattice of $G$. 
\end{remark}

Suppose that $M$ is closed, i.e., compact manifold without boundary. 
For the $S^1$-equivariant  data $(M,\omega,L,\nabla)$, the $S^1$-equivariant 
Riemann-Roch number $RR_{S^1}(M,L)$ is defined 
as the index of the $S^1$-equivariant 
spin$^c$ Dirac operator twisted by $L$. 

\begin{theorem}
If $M$ is a closed symplectic manifold, then 
we have 
$$
RR_{S^1, {\loc}}(M,L)=RR_{S^1}(M,L), 
$$where the right hand side is regarded as a functional on $R(S^1)$, 
$$
RR_{S^1}(M,L) : \C_{(n)}\mapsto RR_{S^1}(M,L)^{(n)}. 
$$
\end{theorem}
\begin{proof}
We show $RR_{S^1, {\loc}}^{(n)}(M,L)=RR_{S^1}(M,L)^{(n)}$ for 
each $n\in \Z$. 
By (2) and (4) in Lemma~\ref{basic properties} we have 
$\ker(D_{S^1}|_{S^1\cdot x})^{(n)}=H^0(S^1\cdot x, (L,\nabla)|_{S^1\cdot x})^{(n)}=0$, 
for each $x\notin X_{\mu,n}$, and hence, by 
{\it shifting trick} and the localization theorem 
for $S^1$-acyclic compatible system~(\cite[Theorem~2.41]{Fujita-Furuta-Yoshida3}), 
we have 
\begin{eqnarray*}
RR_{S^1}(M,L)^{(n)}&=&
RR_{S^1}(M,L\otimes \C_{(-n)})^{(0)}\\
&=&\ind_{S^1}(X_{\mu-n,0},V_{\mu-n,0},W_L\otimes \C_{(-n)}|_{X_{\mu-n,0}})^{(0)}\\
&&+\sum_{k\neq 0}\ind_{S^1}(X'_{\mu-n,k},V'_{\mu-n,k}, W_L\otimes 
\C_{(-n)}|_{X'_{\mu-n,k}})^{(0)}, 
\end{eqnarray*}
where $X'_{\mu-n,k}$ is an $S^1$-invariant relatively compact open neighborhood of 
$(\mu-n)^{-1}(k)\cap M^{S^1}=\mu^{-1}(n+k)\cap M^{S^1}$ and we put 
$V'_k:=X'_k\setminus \mu^{-1}(k)\cap M^{S^1}$. 
On the other hand 
we have $L_{x}=\C_{(k)}$ for each $x\in \mu^{-1}(k)\cap M^{S^1}$ 
by (4) in Lemma~\ref{basic properties}. 
We can apply the vanishing theorem 
(\cite[Theorem~4.1]{Fujita-Furuta-Yoshida3}) and we have 
$\ind_{S^1}(X'_{\mu-n, k},V'_{\mu-n, k}, 
W_L\otimes \C_{(-n)}|_{X'_{\mu-n,k}})^{(0)}=0$ for all $k\neq 0$.  
Note that we may assume that $X_{\mu-n,0}=X_{\mu,n}$ and 
$V_{\mu-n,0}=V_{\mu,n}$. 
So we have $RR_{S^1}(M,L)^{(n)}=
\ind_{S^1}(X_{\mu,n},V_{\mu,n},W_L\otimes \C_{(-n)}|_{X_{\mu,n}})^{(0)}
=RR_{S^1, {\loc}}^{(0)}(M,L\otimes\C_{(-n)})=
RR^{(n)}_{S^1, {\loc}}(M,L)$. 
\end{proof}
\section{Quantization conjecture for $RR_{S^1, {\loc}}$}
Let $(M,\omega), (L,\nabla)$ and $\mu$ be the 
data as in Section~\ref{propermomentmap}. 
Namely $(M,\omega)$ is a symplectic manifold and 
$(L,\nabla)$ is a pre-quantizing line bundle with Hamiltonian 
$S^1$-action whose moment map is $\mu$. 
We assume that 
$\mu^{-1}(n)$ is a compact subset for each $n\in\Z$. 
Suppose that an integer $n$ is a regular value of $\mu$. 
Then we have a closed symplectic orbifold 
$M_{(n)}:=\mu^{-1}(n)/S^1$ with the pre-quantizing line bundle 
$L_{(n)}:=(L\otimes \C_{(-n)},\nabla)|_{\mu^{-1}(n)}/S^1$. 
One can define the Riemann-Roch number $RR(M_{(n)}, L_{(n)})$ 
of the pre-quantized symplectic orbifold $M_{(n)}$ 
as the index of the spin$^c$ Dirac operator twisted by $L_{(n)}$. 
\begin{theorem}
If an integer $n\in \Z$ is a regular value of $\mu$, 
then we have 
$$
RR^{(n)}_{S^1, {\loc}}(M,L)=RR(M_{(n)},L_{(n)}). 
$$
\end{theorem}

\begin{proof}
By the excision formula, the index 
$RR^{(n)}_{S^1,loc}(M,L)=\ind_{S^1}(X_{\mu,n}, V_{\mu,n}, W_L|_{X_{\mu,n}})^{(n)}$ 
is localized at any neighborhood of $\mu^{-1}(n)$. 
On the other hand by the normal form theorem (e.g., 
\cite[Proposition~5.11]{Fujita-Furuta-Yoshida3}) we may assume that the neighborhood has the form
$\mu^{-1}(n)\times_{S^1}T^*S^1$.   
Since the $S^1$-invariant part of the index of $T^*S^1=S^1\times \R$ with 
the standard structure is equal to 1, 
we have $$RR^{(n)}_{S^1}(M,L)
=\ind_{S^1}(X_{\mu,n}, V_{\mu,n}, W_L|_{X_{\mu,n}})^{(n)}
=\ind(\mu^{-1}(n)/S^1, W_{L_{(n)}})=
RR(M_{(n)},L_{(n)})$$ by the product formula. 
\end{proof}

\begin{remark} Kirwan~\cite{Kirwan} and Meinrenken-Sjamaar~\cite{Meinrenken-Sjamaar} 
gave definitions of $RR(M_{(n)},L_{(n)})$ 
for a critical value $n$ of $\mu$. 
We do not understand relation between 
them and $RR^{(n)}_{S^1, {\loc}}(M,L)$. 
\end{remark}

\section{Relation with the transverse index}
Vergne~\cite{Vergne3} gave a formulation of quantization conjecture for 
non-compact symplectic manifolds with Hamiltonian action of a general compact Lie group $G$, in which the compactness of the zero set of the induced vector field (Kirwan vector filed) is assumed. 
Her conjecture concerns with the {\it transverse index} 
which was defined by Atiyah~\cite{Atiyah1} and studied by Paradan-Vergne~\cite{ParadanVergne}. 
Her conjecture was proved by Ma-Zhang~\cite{Ma-Zhang,Ma-Zhangarxiv} and Paradan gave a
new proof in \cite{Paradan2}.  
In \cite{Ma-Zhangarxiv} they defined an equivariant index 
$Q(L):R(G)\to \Z$ under the weaker assumption, the properness of the 
moment map, and showed that the quantization conjecture for $Q(L)$.  
Namely for each irreducible representation $\rho$ of $G$, 
the number $Q(L)(\rho)$ is equal to the Riemann-Roch number of the 
symplectic quotient. 
They used the index theorem due to Braverman \cite{Braverman}.  
He showed that a perturbation of Dirac operator 
gives an analytic realization of the transverse index $\chi_G(M)=\chi_{G}(M,\mu)$. 
The perturbation term is 
the Clifford action of the induced vector field.  
If the induced vector field has compact zero set, 
then the equivariant index $Q(L)$ is equal to the transverse index $\chi_G(M)$. 
Since both equivariant indices $Q(L)$ and $RR_{S^1, {\loc}}(M,L)$ 
satisfy the quantization conjecture, we have the following. 

\begin{prop}
Let $(M,\omega,L,\nabla)$ be a pre-quantized symplectic manifold 
equipped with a Hamiltonian $S^1$-action. 
If the moment map $\mu$ is proper and an integer $n$ is a regular value of $\mu$, then we have 
$$
Q(L)(\C_{(n)})=\chi_{S^1}(M,\mu)(\C_{(n)})=RR_{S^1, {\loc}}^{(n)}(M,L). 
$$
In particular if $\mu$ is proper and it does not have any critical points, 
then we have 
$$
Q(L)=\chi_{S^1}(M,\mu)=RR_{S^1, {\loc}}(M,L)
$$  as functionals on $R(S^1)$. 
\end{prop}

We do not know any direct proof of the second equality 
which does not use the quantization conjecture. 
On the other hand the following example implies that 
our equivariant index $RR_{S^1, {\loc}}(M,L)$ has  
different behaviour from the transverse index. 

\begin{example}
Let $m$ be a non-zero integer and 
$M$ the product of the circle $S^1$ and a small interval centered at $m$. 
Consider the standard metric and 
the symplectic structure on $M$. 
Let $L$ be the trivial complex line bundle over $M$ 
which is equipped with a structure of pre-quantizing line bundle over $M$. 
Consider the natural $S^1$-action on $M$, and we take 
its lift to $L$ so that $S^1$ acts trivially on the fiber direction. 
One has the associated moment map $\mu$ 
which is equal to the projection to the interval factor. 
Since $m$ is non-zero $\mu$ does not have neither critical points nor zeros, 
and hence, the associated vector field $\mu^M$  on $M$ does not vanish. 
Then  \cite[Lemma~3.12]{Braverman} 
implies that the associated transverse index $\chi_{S^1}(M,\mu)$ vanishes. 
(In fact one can check that  
the kernel of the perturbation of the Dirac operator by $\mu^M$ vanishes 
by the direct computation.) 
On the other hand one can check that the kernel 
of the perturbation by $D_{S^1}$ is one dimensional and it is isomorphic to 
$\C_{(n)}$, hence, we have 
$RR^{(n)}_{S^1, {\loc}}(M,L)=\delta_{mn}$. 
In particular we have $RR_{S^1, {loc}}(M,L)\neq \chi_{S^1}(M,\mu)$. 
See Appendix~\ref{appendixa} for details of the computation. 
\end{example}

\appendix
\section{Perturbations on the cylinder and some computations}\label{appendixa}
In this appendix we give some details of the computations of the kernel of 
the perturbed Dirac operator on the cylinder. 
We consider a family of perturbations 
which includes perturbations used in \cite{Braverman}, \cite{Ma-Zhang} and \cite{Fujita-Furuta-Yoshida1}. 
\subsection{Setting}
\begin{enumerate}
\item $m$ : integer 
\item $M=\R\times S^1$ with coordinate functions $(r,\theta)$ 
\item $g=dr^2+d\theta^2$ : Riemannian metric 
\item $\omega=dr\wedge d\theta$ : symplectic structure 
\item $J : \partial_r\mapsto \partial_{\theta}, \quad 
\partial_{\theta}\mapsto -\partial_r $ : almost complex structure 
\item We use $\partial_{\theta}$ as a frame of $TM_{\C}=(TM,J)$. 
\item $W^+=M\times\C$, \ $W^-=TM_{\C}$, \ $W=W^+\oplus W^-$ 
\item $c:T^*M\to {\rm End}(W)$ : Clifford action defined by 
$$c(dr)=
\begin{pmatrix}
0 & -\sqrt{-1} \\ 
-\sqrt{-1} & 0 
\end{pmatrix}, \quad 
c(d\theta)=
\begin{pmatrix}
0 & -1 \\ 
1 & 0
\end{pmatrix} 
$$
\item $\rho : \R\to \R$ : 
smooth non-decreasing function with 
$$
\rho(r)=\left\{ 
\begin{array}{lll}
r \quad (m-1/4 < r < m+1/4)) \\ 
m - 1/2  \quad (r<m-1/2) \\
m + 1/2 \quad (r>m+1/2) 
\end{array}\right.
$$
\item $\nabla^W=d-2\pi\rho(r)
\begin{pmatrix}
1 & 0 \\ 
0 & 1
\end{pmatrix}d\theta$ : 
Clifford connection of $W$ 
\item 
$D:\Gamma(W)\to \Gamma(W)$ : Dirac operator, 
$$
D=
\begin{pmatrix}
0 &  -\partial_{\theta}-\sqrt{-1}\partial_r+2\pi\sqrt{-1}\rho \\
\partial_{\theta}-\sqrt{-1}\partial_r-2\pi\sqrt{-1}\rho & 0
\end{pmatrix}
$$
\item Let $S^1$ acts on $M$ in the standard way, and   
we take a lift of 
the $S^1$-action on $W$ so that the action on the fiber direction is trivial. 
All the data are preserved by the $S^1$-action.  

\item $D_{S^1}:\Gamma(W) \to \Gamma(W)$ : Dirac operator along the $S^1$-orbits : 
$$
D_{S^1}=
\begin{pmatrix}
0 &  -\partial_{\theta}+2\pi\sqrt{-1}\rho \\
\partial_{\theta}-2\pi\sqrt{-1}\rho & 0
\end{pmatrix}
$$
\item $\mu:=-2\pi\rho:M\to \R$
\item $\mu^M=-2\pi\rho\partial_{\theta}\in \Gamma(TM)$ : induced vector field 
\item $f:M\to \R_+$ : smooth positive function on $M$ 
such that $f(r)=|r|$ for $|r-m|>1/2$ 
\end{enumerate}

\begin{remark}
The above data is a completion of the standard Hamiltonian $S^1$-action on the symplectic manifold $(m-1/4, m+1/4)\times S^1$ with pre-quantizing line bundle $(L,\nabla)=(\underline{\C}, d+\sqrt{-1}\mu d\theta)$, which has a cylindrical end 
and translational invariance on the end. 
The map $\mu$ gives the moment map of the $S^1$-action on this symplectic manifold, 
and $f^{\vep}$ is an admissible function for $(W,\mu,\nabla^W)$ for any $\vep>0$ 
in the sense of \cite{Braverman}. 
\end{remark}

\subsection{Perturbation of $D$}

For $s, t, \vep_1, \vep_2 \ge 0$ 
we consider the following perturbation of $D$ :
$$
D_{s, t,\vep_1, \vep_2}:=D+\sqrt{-1}sf^{\vep_1}c(\mu^M)+tf^{\vep_2}D_{S^1}=
\begin{pmatrix}
0 &  D_{s, t, \vep_1, \vep_2}^- \\ 
D_{s, t, \vep_1, \vep_2}^+ & 0  
\end{pmatrix}, 
$$where 
$$
D_{s, t, \vep_1, \vep_2}^+=
(1+tf^{\vep_2})(\partial_{\theta}-2\pi\sqrt{-1}\rho)-\sqrt{-1}\partial_r
-2\pi\sqrt{-1}sf^{\vep_1}\rho  
$$
and 
$$
D_{s, t, \vep_1, \vep_2}^-=
-(1+tf^{\vep_2})(\partial_{\theta}-2\pi\sqrt{-1}\rho)-\sqrt{-1}\partial_r
+2\pi\sqrt{-1}sf^{\vep_1}\rho.   
$$

Note that $D_{1,0,\vep_1,\vep_2}$ ($\vep_1>0$) is the perturbation 
considered in \cite{Braverman} and \cite{Ma-Zhang}, and 
$D_{0,t,\vep_1,0}$ is the one considered in 
\cite{Fujita-Furuta-Yoshida1}. 

\subsection{$\ker_{L^2}(D^+_{s, t, \vep_1, \vep_2})$}
For $\phi\in \Gamma(W^+)$ by taking the Fourier expansion 
we write 
$$
\phi(r,\theta)=\sum_{n\in \Z}a_n(r)e^{2\pi\sqrt{-1}n\theta}. 
$$Then we have 
$$
D^+_{s, t, \vep_1, \vep_2}\phi=
\sum_{n\in\Z}\sqrt{-1}\left(2\pi((1+tf^{\vep_2})(n-\rho)-sf^{\vep_1}\rho)a_n(r)
-a_n'(r)\right)e^{2\pi\sqrt{-1}n\theta}, 
$$and hence, 
\begin{eqnarray*}
D_{s, t, \vep_1, \vep_2}^+\phi=0 &\Longleftrightarrow& 
2\pi((1+tf^{\vep_2})(n-\rho)-sf^{\vep_1}\rho)a_n(r)
-a_n'(r)=0 \\ 
&\Longleftrightarrow& 
a_n(r)=\alpha_n\exp\left(2\pi\int^r((1+tf^{\vep_2})(n-\rho)-sf^{\vep_1}\rho)dr\right) \ (\alpha_n\in \C). 
\end{eqnarray*}
Now we determine the condition for $\phi\in\ker(D_{s, t, \vep_1, \vep_2})$ to be an $L^2$-section. 
Since $\rho=m\pm 1/2$ and $f=|r|$ for $\pm r$ large enough we have
\begin{eqnarray*}
a_n(r)&=&\alpha_n\exp\left(2\pi\int^r((1+ t|r|^{\vep_2})(n-m\mp 1/2)-
(m\pm1/2)s|r|^{\vep_1})dr\right)\\ 
&=&\alpha_n\exp\left(2\pi(n-m\mp 1/2)\left(r+\frac{tr|r|^{\vep_2}}{\vep_2+1}\right)
-2\pi(m\pm1/2)\frac{sr|r|^{\vep_1}}{\vep_1+1}\right). 
\end{eqnarray*}

Suppose that $\phi$ is an $L^2$-solution, i.e., $\int_{-\infty}^{\infty}|a_n(r)|^2dr<\infty$. 

\paragraph{(I) $\vep_1>\vep_2$. }
In this case when we take $r\gg 0$ we have 
$m+1/2>0$, and when we take $-r\gg 0$ we have 
$m-\frac{1}{2}<0$. 
So we have $m=0$, and hence, we have 
$\ker_{L^2}(D_{s, t, \vep_1, \vep_2})\neq 0$ if and only if $m=0$. 
If $m=0$, then $\ker_{L^2}(D_{s, t, \vep_1, \vep_2})$ is 
an infinite dimensional vector space generated by 
$\{a_n(r)e^{2\pi\sqrt{-1}n\theta} \ | \ n\in\Z\}$. 

\paragraph{(II) $\vep_1<\vep_2$. }
In this case as in the same way for (I) 
we have $m-1/2<n<m+1/2$, and hence,  
$\ker_{L^2}(D_{s, t, \vep_1, \vep_2})=\C\langle a_m(r)e^{2\pi\sqrt{-1}m\theta}\rangle$. 

\paragraph{(III) $\vep_1=\vep_2$. }
In this case when we take $r\gg 0$ and $-r\gg 0$ we have 
$$
\left(n-m-\frac{1}{2}\right)t-\left(m+\frac{1}{2}\right)s<0   \ {\rm and}
\left(n-m+\frac{1}{2}\right)t-\left(m-\frac{1}{2}\right)s>0. 
$$
So we have if $t=0$ and $s>0$, then $m=0$, and if $t>0$, then  
$$
\left(1+\frac{s}{t}\right)\left(m-\frac{1}{2}\right)<n< 
\left(1+\frac{s}{t}\right)\left(m+\frac{1}{2}\right). 
$$
In this case $\dim\ker_{L^2}(D_{s,t,\vep_1,\vep_1}^+)$ depends 
on $s/t$.

\subsection{$\ker_{L^2}(D^-_{s, t, \vep_1, \vep_2})$}
For $\phi\partial_{\theta}\in \Gamma(W^-)$ by taking the Fourier expansion 
we write 
$$
\phi(r,\theta)=\sum_{n\in \Z}a_n(r)e^{2\pi\sqrt{-1}n\theta}. 
$$Then we have 
$$
D^-_{s, t, \vep_1, \vep_2}\phi=-
\sum_{n\in\Z}\sqrt{-1}\left(2\pi((1+tf^{\vep_2})(n-\rho)-sf^{\vep_1}\rho)a_n(r)
+a_n'(r)\right)e^{2\pi\sqrt{-1}n\theta}, 
$$and hence, 
\begin{eqnarray*}
D_{s, t, \vep_1, \vep_2}^-\phi=0 &\Longleftrightarrow& 
2\pi((1+tf^{\vep_2})(n-\rho)-sf^{\vep_1}\rho)a_n(r)
+a_n'(r)=0 \\ 
&\Longleftrightarrow& 
a_n(r)=\alpha_n\exp\left(-2\pi\int^r((1+tf^{\vep_2})(n-\rho)-sf^{\vep_1}\rho)dr\right) \ (\alpha_n\in \C). 
\end{eqnarray*}
As in the same way for $\ker_{L^2}(D^+_{s, t, \vep_1, \vep_2})$ 
one can check that  
there are no $L^2$-solutions of $D_{s,t,\vep_1,\vep_2}^-\phi=0$. 
\subsection{Computations of indices}

We specialize the parameters and have computations of 
two indices, the transverse index  $\chi_{S^1}(M,\mu)$ in \cite{Braverman} 
and the equivariant local Riemann-Roch number $RR_{S^1, loc}(M,L)$. 

\subsubsection{$\chi_{S^1}(M,\mu)$}
When we take $s=1$, $t=0$ and $\vep_1>\vep_2$ 
we have the following. 
\begin{prop}
$$
\Ker_{L^2}(D^+_{1,0,\vep_1,\vep_2})=
\C\langle\{\delta_{m0}a_n(r)e^{2\pi\sqrt{-1}n\theta} \ | \ n\in\Z\}\rangle, 
\quad \Ker_{L^2}(D^-_{1,0,\vep_1,\vep_2})=0. 
$$
In particular we have 
$$
\chi_{S^1}(M,\mu)=
\C\langle\{\delta_{m0}a_n(r)e^{2\pi\sqrt{-1}n\theta} \ | \ n\in\Z\}\rangle
=\bigoplus_{n\in\Z}\delta_{m0}\C_{(n)}.  
$$
\end{prop}

\subsubsection{$RR_{S^1, {\loc}}(M,L)$}
When we take $s=\vep_2=0$ 
we have the following. 
\begin{prop}
$$
\Ker_{L^2}(D^+_{0,t,\vep_1,0})=
\C\langle a_m(r)e^{2\pi\sqrt{-1}m\theta}\rangle, 
\quad \Ker_{L^2}(D^-_{0,t,\vep_1,0})=0. 
$$
In particular we have 
$$
RR_{S^1, {\loc}}(M,L)=
\C\langle a_m(r)e^{2\pi\sqrt{-1}m\theta}\rangle
=\C_{(m)}.   
$$
\end{prop}

\section*{Acknowledgements}
The author would like to thank 
Mikio Furuta and Takahiko Yoshida 
for stimulating discussions. 
He is grateful to Michele Vergne and 
Xiaonan Ma for discussions on the relation between our index and 
the transverse index. 

\bibliographystyle{mrl}
\bibliography{reference}

\end{document}